\documentclass{amsart}

\usepackage{amsmath}
\usepackage{amssymb}

\newtheorem{theorem}{Theorem}
\newtheorem{lemma}{Lemma}

\newtheorem*{KT}{\textup{\textbf{Khinchin's Theorem}}}
\newtheorem*{Theorem}{\textup{\textbf{Theorem}}}

\theoremstyle{definition}

\theoremstyle{remark}
\newtheorem{remark}{\textup{\textbf{Remark}}}

\newcommand{\abs}[1]{\lvert#1\rvert}
\newcommand{\absb}[1]{\bigl\lvert#1\bigr\rvert}
\newcommand{\absB}[1]{\Bigl\lvert#1\Bigr\rvert}

\newcommand{\absBB}[1]{\Biggl\lvert#1\Biggr\rvert}
\newcommand{\NN}{\mathbb{N}}
\newcommand{\NNa}{\mathbb{N}^*}
\newcommand{\ZZ}{\mathbb{Z}}

\newcommand{\RR}{\mathbb{R}}
\newcommand{\set}[1]{\{#1\}}
\newcommand{\SET}[2]{\{#1\thinspace |\thinspace #2\}}

\newcounter{simplesection} 
\newcommand{\simplesection}[1]
   {\pagebreak[2]\begin{flushleft}\bf{\arabic{simplesection}\stepcounter{simplesection}. #1}\\ 
     \end{flushleft}}
\begin{document}

\title[Continuous, integrable functions with extreme behavior at infinity
]{Construction of continuous, integrable functions with extreme behavior at infinity}


%    author one information
\author[G. Batten]{George W. Batten, Jr.}
\address{Advanced Image Measurement Systems\\
1211 Rosine Street\\Houston, Texas, 77019}
\curraddr{}
\email{gbatten.aims@gmail.com}
\thanks{}

%    author two information
%\author{}
\address{}
\curraddr{}
\email{}
\thanks{}

%    \subjclass is required by all journals except JAG.
%\subjclass[2000]{Primary }
%    The 2010 edition of the Mathematics Subject Classification is
%    now available.  If you are citing a classification from the
%    new scheme, use the following input coding instead.

\date{}

\dedicatory{}

%    Abstract is required.
\begin{abstract}
For any real sequence $\{c_n\}$ with $c_n\rightarrow\infty$, this constructs a function 
$f$ which is continuous and integrable on the real line, and such that for every real 
$x\neq0$ $\limsup_{n\rightarrow\infty}c_nf(n\,x)=\infty$.
\end{abstract}

\maketitle


Let $f$ be a continuous, integrable, real-valued function on the real numbers $\RR$.   
Let $x\in\RR$, $x\neq0$, and consider what happens to $f(n\,x)$ as $n$ 
tends to $\infty$ through the positive integers $\NNa$.\footnote{$\NN=\{0,1,2,\dotsc\}$, 
and $\NNa=\{1,2,\dotsc\}$, the ISO 31-11 conventions.}   
Emmanuel Lesigne \cite{Lesigne} shows the following:
\begin{enumerate}
\item[T1.] Even if $f$ is not continuous,
$f(n\thinspace x)\rightarrow0$ for almost all $x$.
\item[T2.] For any sequence $\{c_n\}$ with $c_n\rightarrow\infty$, no matter how slowly, 
there is a nonnegative, continuous, integrable function $f_{E'}$ with  
$\limsup_{n\rightarrow\infty}c_nf_{E'}(n\,x)=\infty$ for all $x$ not in a set $E_\textup{K}$ 
of Lebesgue measure zero.
\item[T3.] In this, the condition $c_n\rightarrow\infty$ cannot be replaced by 
$\limsup_{n\rightarrow\infty}c_n = \infty$.
\end{enumerate}
He asks whether the the second statement (T2) is true for \textit{all} $x\neq0$.  
In this paper we show that it is:

\begin{Theorem}\label{th:1}
For any sequence $\SET{c_n}{n\in\NNa}$ with $c_n\rightarrow\infty$ there is a 
continuous, integrable function $f$ such that  
$\limsup_{n\rightarrow\infty}c_nf(n\,x)=\infty$ for all $x\neq0$.
\end{Theorem}

Our proof constructs a nonnegative, continuous, integrable function $f_E$ which satisfies
the theorem for all $x$ in a set $E\supseteq E_{\textup{K}}$.  Then $f:=f_E+f_{E'}$
satisfies the theorem.\footnote{A colon-equals (:=) combination indicates 
a definition of the item on the left side.} 

In the proof T2, Lesigne uses Khinchin's 
Theorem (Theorem 32 of Khinchin's book \cite{Khinchin}; see also, the Appendix, below).  
It is that theorem that provides the set $E_\textup{K}$.

Khinchin's proof establishes that
 $E_\textup{K}\cap (0,1)\subset E_F^0\cup E_G^0$, where
$E_F^0$ and $E_G^0$ (our notation) are sets constructed in the proofs of the book's 
Theorem 30 and 31, respectively, the superscript $0$ indicating subsets of (0,1).
Khinchin shows that $E_F^0$ and $E_G^0$ are sets of measure zero, and he
notes that the result applies to all intervals of $\RR$ through translation by integers;
i.e., that $E_\textup{K}\subseteq \cup_{j=-\infty}^\infty (j+E_F^0\cup E_G^0)$.

The translation technique does not seem to apply here.  If we knew that the theorem were true
for $x\in (0,1)$, then extending it would require dealing with $f(n\,(x+j)) = f(n\,x+n\,j)$
where $j\in\ZZ$, the set of integers, and it is not clear how one can relate this to $f(n\,x)$.
Therefore, we will take a different approach.

In the proofs of Theorems 30 and 31 of his book, Khinchin defines open subsets 
$E_{m,n}$ and $E_n(e^{A\,n})$ of $(0,1)$.
This notation is not well suited to our purposes, 
so we will use $F_{mn}^0$ and $G_{n}^0$ for these sets, and we have 
$E_F^0:=\cup_{m=1}^\infty \cap_{k=m}^\infty F_{mk}^0$ and
$E_G^0:=\cap_{k=1}^\infty G_{k}^0$.
 
\pagebreak %\newpage %%%%%%%%%%%%%%%%%%%%%%%%%%%
\simplesection{The sets $F_{mk}$ and $G_m$}

For any set $S\subseteq\RR$, we will use $\overline{S}$ and $\abs{S}$ for the closure 
and Lebesgue measure of $S$, respectively.
Under the conditions that Khinchin imposes, which we assume, we have sets with the following properties:
\begin{equation}\label{eq:S3}
\absb{F_{mk}^0} \rightarrow 0 \text{ as } k \rightarrow \infty, \quad
    \absB{\overline{F_{mk}^0}}=   \absb{F_{mk}^0}\ne 0, \quad
    0\notin F_{mk}^0;
\end{equation}
\begin{equation}\label{eq:S5}
\absb{G_k^0} \rightarrow 0 \text{ as } k \rightarrow \infty, \quad
    \absB{\overline{G_k^0}}=   \absb{G_k^0}\ne 0, \quad
    0\notin G_k^0.
\end{equation}
Discussion of these properties is deferred until the Section 5.

Extend $F_{mk}^0$ and $G_k^0$ to all intervals in $\RR$ as follows.  
For positive integers $m$ and $k$, 
let $\SET{n(j)\in\NNa}{j\in\ZZ}$ satisfy both
\begin{equation*}
    \absB{F_{m,k+n(j)}^0}<2^{-(\abs{j}+k)}, \text{ and}\quad
    \abs{G_{k+n(j)}^0}<2^{-(\abs{j}+k)},
\end{equation*}
which is possible because of the first property in each of 
\eqref{eq:S3} and \eqref{eq:S5}. 
Translate the sets and combine them: let
\begin{equation*}
    F_{mk}:=\bigcup_{j=-\infty}^\infty (j+F_{m,k+n(j)}^0), \quad\text{and}\quad
    G_{k}:=\bigcup_{j=-\infty}^\infty (j+G_{k+n(j)}^0).
\end{equation*}
Since these are countable unions of sets in pairwise disjoint intervals, from the second property of each 
of \eqref{eq:S3} and \eqref{eq:S5},
$\absb{\overline{F_{mk}}}=\abs{F_{mk}}$, and
$\absb{\overline{G_k}} = \abs{G_k}$.
Also
\begin{equation*}%\label{eq:S9}
\absb{F_{mk}} = \sum_{j=-\infty}^\infty\absb{j+F_{m,k+n(j)}^0} <  
\sum_{j=-\infty}^\infty2^{-(\abs{j}+k)} = 3\cdot 2^{-k}.
\end{equation*}
The same applies to $G_k$, so $\abs{G_k}<3\cdot 2^{-k}$. 
Thus, we have the following extensions of \eqref{eq:S3} and \eqref{eq:S5}:
\begin{equation}\label{eq:S11}
\absb{F_{mk}} \rightarrow 0 \text{ as } k \rightarrow \infty, \quad
    \absB{\overline{F_{mk}}}=   \absb{F_{mk}}, \quad
    0\notin F_{mk};
\end{equation}
\begin{equation}\label{eq:S13}
\absb{G_k} \rightarrow 0 \text{ as } k \rightarrow \infty, \quad
    \absB{\overline{G_k}}=   \absb{G_k}, \quad
    0\notin G_k.
\end{equation}

Let 
\begin{equation*}
    F_m := \bigcap_{k=m}^\infty F_{mk},\quad
    F := \bigcup_{m=1}^\infty F_m, \text{ and}\quad
    G := \bigcap_{k=1}^\infty G_k. 
\end{equation*}
Then $\abs{F_m} = 0,\,\abs{F} = 0, \text{ and } \abs{G}=0$.
Let $E:=F\cup G$.  Then $\abs{E}=0$.
It is easy to establish that
$E \supseteq \cup_{j=-\infty}^\infty (j+E_F^0\cup E_G^0)$, so 
$E\supseteq E_\textup{K}$.

We will construct nonnegative, bounded, continuous functions $f_F$ and $f_G$ which are 
zero except on a set of finite measure, and such that, for sufficiently large $n$,
$ f_F(n\,x)\ge 1/\sqrt{\smash[b]{c_n}}$ when $x\in F$, and $ f_G(n\,x)\ge 1/\sqrt{\smash[b]{c_n}}$ 
 when $x\in G$ .
Then $f_E:=f_F+f_G$ will be integrable and $\lim_{n\rightarrow\infty}c_nf_E(n\,x)=\infty$ 
for $x\in E$.  This is more than is necessary.

%\newpage %%%%%%%%%%%%%%%%%%%%%%%%%%%
\simplesection{Auxiliary functions}

Here we define several functions used in constructing $f_F$ and $f_G$.
The continuous function
\begin{equation*}
u(x):=\left\{
    \begin{aligned}
        &0              &&\text{ for } \abs{x}<1/2,\\
        &2\,\abs{x}-1 &&\text{ for } 1/2\le\abs{x}<1,\\
        &1               &&\text{ for } 1\le \abs{x}
    \end{aligned}
\right.
\end{equation*}
will be used to restrict the support of functions to be defined.

For each $m\in\NNa$, choose increasing sequences 
$\SET{k(m,l)\in\NNa}{k(m,l)\ge m,l\in\NNa}$ and
$\SET{k(l)\in\NNa}{l\in\NNa}$,
so that $k(m,l)\rightarrow\infty$ and $k(l)\rightarrow \infty$ as $l\rightarrow\infty$; 
\begin{equation*}
    \absb{F_{m,k(m,l)}}<\frac{1}{l}\thinspace 2^{-(m+l+1)}; \text{ and}\quad
    \absb{G_{k(l)}}<\frac{1}{l}\thinspace 2^{-(l+1)}.
\end{equation*}
This is possible because of the first property in each of 
\eqref{eq:S11} and \eqref{eq:S13}.  
Let $F_{m,k(m,l)}^*$ and $G_{k(l)}^*$ be an open sets such that
\begin{enumerate}
\item[1.] $\overline{F_{m,k(m,l)}} \subset F_{m,k(m,l)}^*, \text{ and}\quad
                 \overline{G_{k(l)}} \subset G_{k(l)}^*$;
\item[2.] $\abs{F_{m,k(m,l)}^*}<2\,\abs{F_{m,k(m,l)}}, \text{ and}\quad
                 \abs{G_{k(l)}^*}<2\,\abs{G_{k(l)}}$.
\end{enumerate}
For example, because of the second property in each of \eqref{eq:S11} and \eqref{eq:S13}, 
we can take $F_{m,k(m,l)}^*$ and 
$G_{k(l)}^*$ to be open sets determining the outer measure of $\overline{F_{m,k(m,l)}}$ and 
$\overline{G_{k(l)}}$, respectively.
Note that 
\begin{equation*}
    F_m \subseteq \bigcap_{l=1}^\infty F_{m,k(m,l)},\text{ and}\quad
    G \subseteq \bigcap_{l=1}^\infty G_{k(l)}.
\end{equation*} 

Let $v_{Fml}$ and $v_{Gl}$ be continuous functions with values in $[0,1]$, and
\begin{equation*}
v_{Fml}(x):=
  \begin{cases}
     0 \text{ for } x\notin F_{m,k(m,l)}^*\\
     1 \text{ for } x\in \overline{F_{m,k(m,l)}},
  \end{cases}
\end{equation*}
and
\begin{equation*}
v_{Gl}(x):=
  \begin{cases}
     0 \text{ for } x\notin G_{k(l)}^*\\
     1 \text{ for } x\in \overline{G_{k(l)}}.
  \end{cases}
\end{equation*}
Urysohn's Lemma provides such functions. 
We see that $v_{Fml}(x)=1$ for $x\in F_m$, and $v_{Gl}(x)=1$ for $x\in G$.

%\newpage %%%%%%%%%%%%%%%%%%%%%%%%%%%
\simplesection{The functions $f_{Fm}$ and $f_{Gm}$}

Henceforth assume, without loss of generality, that 
$\{c_n\}$ is positive and nondecreasing (otherwise, 
replace any $c_n\le 0$ with 1, a finite number of replacements since $c_n\rightarrow\infty$, then replace every
$c_n$ with $\inf_{k\ge n}c_k$).

For $x\in\RR$, and $m\in\NNa$, let
\begin{equation*}
f_{Fm}(x):=\sup_{l\in\NNa}\Biggl(
                                                  \frac{1}{\sqrt{c_l}}\, 
                                                  v_{Fml}\Bigl(\frac{x}{l}\Bigr)\,
                                                  u\Bigl(l^{-1}m\,x^2\Bigr)\, 
                                                  u\Bigl(\frac{x}{m}\Bigr)
                                       \Biggr),
\end{equation*}
and
\begin{equation*}
f_{Gm}(x):=\sup_{l\in\NNa}\Biggl(
                                                  \frac{1}{\sqrt{c_l}}\, 
                                                  v_{Gl}\Bigl(\frac{x}{l}\Bigr)\,
                                                  u\Bigl(l^{-1}m\,x^2\Bigr)\, 
                                                  u\Bigl(\frac{x}{m}\Bigr)
                                       \Biggr),
\end{equation*}
\pagebreak[3]
These functions have the following properties:
\pagebreak[3]
\begin{enumerate}
\item[P1.] $0\le f_{Fm}(x)\le 1/\sqrt{c_1}$, and $0\le f_{Gm}(x)\le 1/\sqrt{c_1}$.
\item[P2.] $f_{Fm}$ and $f_{Gm}$ are continuous.
\item[P3.] If $n\in\NNa$, $m\,n\,x^2 \ge 1$,
                          and $n\,\abs{x} \ge m$; then 
                          $f_{Fm}(n\,x)\ge 1/\sqrt{c_n}$ if $x\in F_m$, and
                          $f_{Gm}(n\,x)\ge 1/\sqrt{c_n}$ if $x\in G$.
\item[P4.] For $x$ in any bounded set, $f_{Fm}(x)=0$ and $f_{Gm}(x)=0$ except for 
                           a finite number of values of $m$.
\item[P5.] $\SET{x}{f_{Fm}(x)\ne 0}\subseteq
                           \bigcup_{l=1}^\infty (l\,F_{m,k(m,l)}^*)$, and
                $\SET{x}{f_{Gm}(x)\ne 0}\subseteq
                           \bigcup_{l=1}^\infty (l\,G_{k(l)}^*)$.
\item[P6.] $\bigl\lvert \SET{x}{f_{Fm}(x)\ne 0}\bigr\rvert \le 2^{-m}$, and
                $\bigl\lvert \SET{x}{f_{Gm}(x)\ne 0}\bigr\rvert \le 2^{-m}$.
\end{enumerate}
The justifications for these are the following:
\begin{enumerate}
\item[1.] $1/\sqrt{\smash[b]{c_l}} \le 1/\sqrt{\smash[b]{c_1}}$ since 
                 $c_n$ is nondecreasing.
\item[2.] Locally each of $f_{Fm}$ and $f_{Gm}$ is the maximum of a 
                 finite number of continuous functions: 
                 if $\abs{x}<M$, then $u(l^{-1}m\,x^2)=0$ if $l>2\,m\,M^2$.
\item[3.] When $l=n$, $n\,x/l=x$, so the following hold: 
                 $l^{-1}m\,(n\,x)^2 = m\,n\,x^2 \ge 1$ so 
                      $u\bigl(l^{-1}m\,(n\,x)^2\bigr) = 1$; 
                $u(n\,x/m)=1$; 
                 if $x\in F_m$, $v_{Fml}(n\,x/l)=1$, so
                      $f_{Fm}(n\,x)\ge 1/\sqrt{c_l} =1/\sqrt{c_n}$; and  
                 if $x\in G$, $v_{Gl}(n\,x/l)=1$, so
                      $f_{Gm}(n\,x)\ge 1/\sqrt{c_l} =1/\sqrt{c_n}$.
\item[4.] For $\abs{x}<M$,  $u(x/m)=0$ for $m\ge 2\,M$. 
\item[5.] $\SET{x}{v_{Fml}(x/l)\ne 0} \subseteq l\, F_{m,k(m,l)}^*$, and 
              $\SET{x}{v_{Gm}(x/l)\ne 0} \subseteq l\, G_{k(l)}^*$.\smallskip
\item[6.] $\bigl\lvert \SET{x}{f_{Fm}(x)\ne 0} \bigr\rvert
                          \le \sum_{l=1}^\infty l\,\absb{F_{m,k(m,l)}^*} 
                          < \sum_{l=1}^\infty  l\cdot 2\,\abs{F_{m,k(m,l)}}\smallskip\\*
             \indent \qquad\qquad\qquad\,
                          < \sum_{l=1}^\infty l\cdot 2 \cdot l^{-1}\,2^{-(m+l+1)} = 2^{-m}$,\\*
            and similarly for $f_{Gm}$.
\end{enumerate}

%\newpage %%%%%%%%%%%%%%%%%%%%%%%%%%%
\simplesection{Proof of the theorem}

Let
\begin{equation*}%\label{eq:PT3}
f_F(x):=\sup_{m\in\NNa} f_{Fm}(x),\text{ and}\quad f_G(x):=\sup_{m\in\NNa} f_{Gm}(x).
\end{equation*}
By P1, P2, and P4, $f_{F}$ is nonnegative, continuous and bounded.
Moreover
\begin{equation*}
\SET{x}{f_{F}(x)\ne 0} = \bigcup_{m=1}^\infty\SET{x}{f_{Fm}(x)\ne 0}
\end{equation*}
so, by P6, 
\begin{equation*}
\bigl\lvert \SET{x}{f_{F}(x)\ne 0} \bigr\rvert < \sum_{m=1}^\infty 2^{-m}=1.
\end{equation*}
Therefore, $f_{F}$ is bounded, and it is zero except on a set of finite measure, hence, 
it is integrable on $\RR$.  
That $f_F(n\,x) \ge 1/\sqrt{c_n}$ for $x \in F$ and large enough $n$ follows from
property P3 since $x$ will be in some $F_m$. 
Thus
\begin{equation*}
\lim_{n\rightarrow\infty} c_n\, f_{F}(n\,x) \ge
\lim_{n\rightarrow\infty} \sqrt{c_n}=\infty.
\end{equation*}
The same applies, \textit{mutatis mutandis}, to $f_G$.

Thus, $f=f_E+f_{E'}=f_F+f_G+f_{E'}$ satisfies the requirements and this completes the proof
of the theorem.\qed

%\newpage %%%%%%%%%%%%%%%%%%%%%%%%%%%
\simplesection{Properties of $F_{mk}^0$ and $G_k^0$}

Khinchin establishes the first property in both \eqref{eq:S3} and \eqref{eq:S5}.  
Proof of other properties in \eqref{eq:S3} and \eqref{eq:S5} 
requires understanding the definitions of the sets $F_{mk}^0$ and $G_m^0$, 
for which we need a few facts about continued fractions.

The notation $\alpha(n)=[a_0;a_1,a_2,\dotsc,a_n]$ will be used for the 
\textit{terminating continued fraction}
\begin{equation*}%\label{eq:CFn}
\alpha(n) = a_0+\cfrac{1}{a_1+
             \cfrac{1}{a_2+\underset{\ddots}{}
               \underset{\underset{\genfrac{}{}{0pt}{0}{}{+\dfrac{1}{a_n}}}{}}{} }}
\end{equation*}
with every numerator equal to $1$.  
The numbers $a_0, \dotsc, a_n$ are called the \textit{elements} of $\alpha(n)$. 
Henceforth, 
$a_0\in\ZZ$;  $a_k\in\NNa$ for $1\le k<n$; and $a_n>1$.  The latter is 
done for uniqueness without loss of generality: if $a_n=1$, remove 
it and replace $a_{n-1}$ with $a_{n-1}+1$. 
A \textit{nonterminating continued fraction} $\alpha$ 
has the same form except that the term with $a_n$ is absent.
It will be convenient to extend the notation as follows:  for a terminating 
$\alpha(n)$, let $a_i=0$ for $i>n$, and use the nonterminating form
\begin{equation*}
    \alpha(n)=[a_0;a_1,\dotsc,a_n,a_{n+1},\dotsc]=[a_0;a_1,\dotsc,a_n,0,0,\dotsc].
\end{equation*}
In particular, if $a_i=0$ for $i\ge 1$, $[a_0;0,0,\dotsc]=a_0$.

A terminating continued fraction is a simple fraction:
\begin{equation}\label{eq:CF0}
\alpha(n) = \frac{p_n}{q_n};
\end{equation}
$p_i$ and $q_i$ are integers defined by
\begin{equation}\label{eq:CF1}
\begin{aligned}
p_{-1}:=1,\quad p_0:=a_0,\quad q_{-1}:=0,\quad q_0:=1,\\
\left.
\begin{aligned}
p_i&:=a_i\,p_{i-1}+p_{i-2}\\
q_i&:=a_i\,q_{i-1}+q_{i-2}
\end{aligned}
\right\}\quad\text{for } i=1,\dotsc
\end{aligned}
\end{equation}
\cite[Theorem 1]{Khinchin}.  
For $i\ge 1$ they are increasing functions of $i$ when $a_i\ne 0$, and
\begin{equation*}
\frac{p_i}{q_i}=[a_0;a_1,\dotsc,a_i]
\end{equation*}
is the \textit{convergent of rank} $i$ of $\alpha(n)$.
We identify any $\alpha \in\RR$ with its unique continued fraction; 
if $\alpha$ is irrational, the continued fraction is nonterminating
and its convergents tend to $\alpha$. 

For any $\alpha$, terminating or not,  the representation
can be truncated:  for $i\ge 1$, $\alpha = \alpha(i,r_i)$, where
\begin{equation}\label{eq:CF1a}
\alpha(i,r_i):=[a_0;a_1,\dotsc,a_{i},r_i],
\end{equation}
with  $r_i:=a_{i+1}+[0;a_{i+2},\dotsc]$.  If $r_{i+1}\ne 0$ (i.e., $a_{i+2}\ne 0$), 
$r_i=a_{i+1}+1/r_{i+1}$, so $a_{i+1}< r_i < a_{i+1}+1$.  
If $\alpha$ is nonterminating, $r_i$ can be any number in $(1,\infty)$. 
From \eqref{eq:CF0} and \eqref{eq:CF1} we have
\begin{equation}\label{eq:CF3}
\alpha(i,r_i)=\frac{r_i\,p_{i}+p_{i-1}}{r_i\,q_{i}+q_{i-1}}.
\end{equation}

In \eqref{eq:CF1}, multiply $p_i$ by $q_{i-1}$ and $q_i$ by $p_{i-1}$, then 
subtract the results to obtain a recursion which yields
\begin{equation}\label{eq:CF2}
    p_{i-1}\,q_i-p_i\,q_{i-1} = (-1)^i \quad\text{for } i\ge 0
\end{equation}
\cite[Theorem 2]{Khinchin}.

We will limit our attention to the interval $(0,1)$; that is, take $a_0=0$. 
According to \eqref{eq:CF3}, the range of $\alpha(i,r_i)$ for $r_i\in(1,\infty)$ is an open  interval 
$J_{\alpha(i)}$ with endpoints
\begin{equation*}
\frac{p_{i}}{q_{i}}, \quad\text{and}\quad \frac{p_{i}+p_{i-1}}{q_{i}+q_{i-1}}.
\end{equation*}
This interval is all numbers of the form $[0;a_1,\dotsc,a_i,a_{i+1},\dotsc]$, with
specified $\alpha(i) =[0;a_1,\dotsc,a_i]$, and any $a_{i+1}\in \NNa$, whether 
terminating or not.  Thus $0\notin J_{\alpha(i)}$. 

Let $\SET{\phi(i)\in\RR}{i\in\NNa \text{ and } \phi(i)>1}$ be a nondecreasing sequence, and
use it to constrain the elements of $\alpha(m+k,r_{m+k})$  by the condition
\begin{equation}\label{eq:E1}
    a_{m+i}<\phi(m+i)\text{ for }i=1,\dotsc,k;
\end{equation}
there are no constraints on $a_1,\dotsc,a_m$ (see Theorem 30 in Khinchin's book). 
The set of all numbers in (0,1) satisfying condition \eqref{eq:E1} and $1<r_{m+k}<\infty$ 
is the open set
\begin{equation}\label{eq:F3}
F_{mk}^0 :=\bigcup_{\alpha(m+k)}J_{\alpha(m+k)},
\end{equation}
where the union is over all $\alpha(m+k)$ satisfying \eqref{eq:E1}.  
Thus $0\notin F_{mk}^0$, and $\abs{F_{mk}^0}\ne 0$.

For $k\in\NNa$ and $A$ a positive number satisfying $A-\ln A - \ln 2 - 1 > 0$, the set of numbers 
$\alpha(k,r_k)\in(0,1)$ satisfying the product bound 
\begin{equation}\label{eq:G1}
    a_1\cdots a_k \ge e^{A\,k},
\end{equation}
and $1<r_k<\infty$ is
\begin{equation*}%\label{eq:G3}
G_{k}^0:=\bigcup_{\alpha(k)} J_{\alpha(k)},
\end{equation*}
where the union is over all $\alpha(k)$ satisfying \eqref{eq:G1} 
(see the proof of Theorem 31 in Khinchin's book).
Thus $0\notin G_{k}^0$, and $\abs{G_{k}^0}\ne 0$.

It is not true that $\absb{\overline{S}}=\abs{S}$ for any set $S\subset\RR$, even if
$S$ is a countable union of pairwise disjoint open intervals.  The construction of a Cantor
set with nonzero measure provides a counterexample.

Here we have special sets, $S = F_{mk}^0$ or $S = G_k^0$, which are countable
unions of pairwise disjoint open sets.  The closure $\overline{S}$ is the union of the
closures of the open sets and a set $L$.  The latter comprises the limits of sequences
in $S$ but not ultimately in a single one of the open sets.  Now we will establish that
$L$ is countable, and that $\abs{\overline{S}} = \abs{S}$.  The continued-fraction
construction provides an ordering of the open intervals that makes this possible.

For this, we consider a sequence $\alpha_n=[0;a_{n1},a_{n2},\dotsc]\in S$,
and let that $\alpha_n\rightarrow \alpha = [0;a_1,a_2,\dotsc]$ as 
$n\rightarrow \infty$.  We will use the truncated notations 
\begin{equation*}
  \alpha_n=[0;a_{n1},\dotsc,a_{ni},r_{ni}] \quad\text{and}\quad
  \alpha=[0;a_{1},\dotsc,a_{i},r_{i}].
\end{equation*}

It is apparent that, for small values of $j$, the condition $\alpha_n\rightarrow\alpha$  
forces the elements $a_{nj}$ to become independent of $n$ as $n$ becomes large.  The 
following lemma is a precise statement of that.

\begin{lemma}\label{le:1}      % --------- LEMMA 1 ---------
If, for some $i\in\NNa$, $a_{n1},\dotsc,a_{ni}$ are bounded, and there are numbers $R_{i,min}$ and
$R_{i,max}$ such that $1<R_{i,min}\le r_{ni}\le R_{i,max}<\infty$,
then there is a number $n_0\in\NNa$ such that, for $1\le j\le i$, $a_{nj}$ is independent of $n$ for $n>n_0$, 
and $a_j=\lim_{\nu\rightarrow\infty} a_{\nu j}$; for sufficiently large $n$, $a_j=a_{nj}$.
\end{lemma}
\begin{proof} 
The lower bound on $r_{ni}$ guarantees that $\alpha_n$ does not terminate before $a_{n,i+2}$,
and $1\le a_{n,i+1}\le r_{ni} \le  a_{n,i+1}+1$.  

If, for a moment, we regard $a_{1},\dotsc,a_{i},\text{ and }r_{i}$ as real variables rather than 
integers, we can differentiate $\alpha_n$ with respect to $a_{ni}$, using \eqref{eq:CF1}, 
\eqref{eq:CF3}, and \eqref{eq:CF2}:
\begin{equation*}
    \absBB{\frac{\partial\,\alpha_{n}}{\partial\,a_{ni}}}=
             \absBB{\frac{r_{ni}^2\,(p_{n,i-1}\,q_{ni}-p_{ni}\,q_{n,i-1})}{(r_{ni}\,q_{ni}+q_{n,i-1})^2}}                    
           = \absBB{\frac{r_{ni}^2\,(-1)^i}{(r_{ni}\,q_{ni}+q_{n,i-1})^2}}
           \ge \epsilon_i>0,
\end{equation*}
where $\epsilon_i = (q_{ni}+q_{n,i-1}/R_{i,min})^{-2}$.

Therefore, again considering the variables to be integers, if $a_{ni}\ne a_{n+1,i}$, then
$\abs{a_{ni}-a_{n+1,i}}\ge 1$, so $\abs{\alpha_n-\alpha_{n+1}}\ge\epsilon_i$.
This cannot happen for large $n$ since $\alpha_n\rightarrow\alpha$, so for sufficiently large $n$, 
$a_{ni}$ must be independent of $n$.  

This can be applied to smaller values of $i$, so all of $a_{n1},\dotsc,a_{ni}$ are independent 
of $n$ for sufficiently large $n$; i.e., for  $n>n_0$, this determining $n_0$.

Since $\alpha_n$ does not terminate at $a_{ni}$, $\alpha_n\neq p_{ni}/q_{ni}$, and we can
solve \eqref{eq:CF3} for $r_{ni}$:
\begin{equation*}
    r_{ni}=\frac{-\alpha_n\,q_{n,i-1}+p_{n,i-1}}{\alpha_n\,q_{ni}-p_{ni}}.
\end{equation*}
Since $r_{ni}$ is bounded, this is a continuous function of $\alpha_n$, so
$r_{ni}$ approaches a limit $r_i^*$, 
and $\alpha = \lim_{\nu\rightarrow\infty}\alpha_{\nu}= [a_{n1},\dotsc,a_{ni},r_i^*]$ for $n>n_0$.
As above, $\abs{a_{nj}-a_j}$ must be less than $1$ for large $n$, so 
$a_j=\lim_{\nu\rightarrow\infty} a_{\nu j}=a_{nj}$ for $1\le j\le i$ if $n>n_0$.
\end{proof}

\pagebreak[3]
\begin{lemma}\label{le:2}        % --------- LEMMA 2 ---------
For some $i\in\NNa$, let $a_{n1},\dotsc,a_{ni}$ be bounded, 
and let $R_{i,min}$ and $R_{i,max}$ be such that 
$1<R_{i,min}\le r_{ni}\le R_{i,max}<\infty$.  
If there is a number $n_0\in\NNa$, and a function $h$ with $0\le h(a_{n1},\dotsc,a_{ni})$ 
for all $n> n_0$, then $0\le h(a_{1},\dotsc,a_{i})$.

Also, for $1\le j\le i$, if there is a number $M_j$ such that, $a_{nj} < M_j$ 
for $n\ge n_0$, then $a_j < M_j$.
\end{lemma}
\begin{proof}
This follows from Lemma 1 because $a_j=a_{nj}$ for sufficiently large $n$.
\end{proof}

If $\set{r_{ni}}$ is not bounded, we cannot conclude that $a_{ni}$ is independent of $n$. 
Indeed, it might be that one of $r_{2\,n,i}$ and $r_{2\,n+1,i}$ tends to $1$, the
other to $\infty$, in which case $a_{2\,n+1,i}-a_{2\,n,i}$ tends to $1$ or $-1$.
For example, $\alpha_{2n}=[0;2,n]$ and $\alpha_{2n+1}=[0;1,1,n]$, 
both of which tend to $[0;2]=1/2$, and for which $r_{2n,1}=n\rightarrow\infty$,
and $r_{2n+1,1}=1+1/n \rightarrow 1$.

\begin{lemma}\label{le:3}      % --------- LEMMA 3 ---------
For some $i\in\NNa$, let $a_{n1},\dotsc,a_{ni}$ be bounded.  
If $r_{ni}$ is not bounded away from $1$ or $\infty$, then then $\alpha$ terminates; i.e, 
$\alpha = [0;a_1,\dotsc,a_j]$ for some $j\le i$.
If $a_1\rightarrow\infty$, $\alpha = 0$ (corresponding to $i=j=0$).
\end{lemma}
\begin{proof} 
From Lemma \ref{le:1}, $a_{n1},\dotsc,a_{n,i-1}$ are constant for sufficiently large $n$.  
For those values of $n$, $p_{i-2}:=p_{n,i-2}$, $p_{i-1}:=p_{n,i-1}$, $q_{i-2}:=q_{n,i-2}$, 
and $q_{i-1}:=q_{n,i-1}$ are independent of $n$, so
\begin{equation*}
    \alpha_n = \frac{r_{n,i-1}\, p_{i-1}+p_{i-2}}{r_{n,i-1}\, q_{i-1}+q_{i-2}}.
\end{equation*}

Suppose $r_{ni}$ is not bounded away from $\infty$. 
Since ${a_{ni}} $ is a bounded integer, and therefore has only a finite number of limit points, there is 
a subsequence of ${\alpha_n}$ with $r_{ni}\rightarrow\infty$, and 
$a_{ni}$ converging to one of those limit points; call the limit point $a_i^*$. 
For that subsequence, $r_{n,i-1}= a_{ni}+1/r_{n,i}\rightarrow a_i^*$, so 
\begin{equation*}
    \alpha = \frac{a_i^*\, p_{i-1}+p_{i-2}}{a_i^*\, q_{i-1}+q_{i-2}},
\end{equation*}
Thus, $\alpha=[0;a_1,\dotsc,a_{i-1},a_i^*]$. If $a_i^*=1$, by our convention we change this
to $\alpha=[0;a_1,\dotsc,a_{i-2},a_{i-1}+1]$.

If $r_{ni}$ is not bounded away from $1$, the same applies except that we use a
subsequence with $r_{ni}\rightarrow 1$, so $r_{n,i-1}=a_{ni}+1/r_{ni}\rightarrow a_{ni}+1$.
In this case $\alpha=[0;a_1,\dotsc,a_{i-1},a_{i}+1]$. 

That $\alpha = 0$ if $a_1\rightarrow \infty$ is obvious.
\end{proof}

\begin{lemma}\label{le:4}   % -------- Lemma 4 ----------
$\absB{\overline{F_{mk}^0}}=\abs{F_{mk}^0}$, and $\absB{\overline{G_{k}^0}}=\abs{G_{k}^0}$.
\end{lemma}
\begin{proof}
Equation \eqref{eq:F3} shows that the set $F_{mk}^0$ is the union of a 
countable set of pairwise disjoint open intervals $J_{\alpha(m+k)}$.
We will show that there is a countable set $L_{mk}$ such that 
$\overline{F_{mk}^0} = L_{mk}\bigcup\Bigl(\bigcup_{\alpha(m+k)}\overline{J_{\alpha(m+k)}}\Bigr)$. 
Since $\absb{\overline{J_{\alpha(m+k)}}}=\abs{J_{\alpha(m+k)}}$, it follows that
$\absB{\overline{F_{mk}^0}}
    =\abs{L_{mk}} + \sum_{\alpha(m+k)} \absb{\overline{J_{\alpha(m+k)}}}
    =0 + \sum_{\alpha(m+k)} \abs{J_{\alpha(m+k)}}
    =\abs{F_{mk}^0}$.

Apply the results above with $S=F_{mk}^0$ and $M_j=\phi(j)$ for $m+1\le j\le m+k$. 
If $a_{n1},\dotsc,a_{nm}$ are bounded, and $r_{n,m+k}$ is bounded away from
$1$ and $\infty$, then $\alpha$ satisfies \eqref{eq:E1} by Lemma \ref{le:2}, 
so it is in some $J_{\alpha(m+k)}$, hence in $\overline{J_{\alpha(m+k)}}$. 
Limits of the remaining sequences constitute the set $L_{mk}$.  
By Lemma \ref{le:3}, each point in this set is zero or it has the form $[0;a_1,\dotsc,a_j]$,
where $1\le j\le m+k$. 
Therefore, $L_{mk}$, is
countable because the set of combinations $\set{a_1,\dotsc,a_{m+k}}$ is countable. 

For $S=G_{k}^0$, let $h(a_1,\dotsc,a_k)=a_1\cdots a_k - e^{A\,k}$.
If $a_{n1},\dotsc,a_{nk}$ are bounded, and $r_{mk}$ is bounded away from $1$ and
$\infty$, apply the Lemma \ref{le:2} to see that $\alpha$ satisfies \eqref{eq:G1},
so $\alpha$ is in some $J_{\alpha(k)}$, hence in $\overline{J_{\alpha(k)}}$.  
As above, limits of the remaining sequences constitute a set $L_{k}$
which is countable, and it follows that $\absB{\overline{G_k^0}}=\abs{G_k^0}$.
\end{proof}
Thus we have established the second and third properties in 
both \eqref{eq:S3} and \eqref{eq:S5}.

%\newpage %%%%%%%%%%%%%%%%%%%%%%%%%%%
\simplesection{Appendix}

For convenient reference we provide
\begin{KT}\cite[Theorem 32]{Khinchin}
Let $\SET{b_n>0}{n\in\NNa}$ have $n\,b_n$ nonincreasing for sufficiently large $n$, and
$\sum_n b_n$ divergent.
Then for almost all $\alpha \in \RR$ there are infinitely-many pairs 
$\{m,n\}\subset\NNa$ such that
\begin{equation}\label{eq:A19}
\abs{n\,\alpha-m}<b_n.\footnote{Khinchin uses the Borel-Cantelli
Lemma to show also that if 
$\sum_n b_n$ converges, then \eqref{eq:A19} is satisfied for only a finite number
of pairs $\set{m,n}$, except for $\alpha$ in a set of measure zero.}
\end{equation}
\end{KT}

\begin{remark}
It is not clear how the elements of $\alpha\in(j,j+1)$ change when 
$\alpha$ is negated.  For $j\ge 0$, the negative of $[j;a_1,a_2,\dotsc]$ is 
$-j-[0;a_1,a_2,\dotsc] = [-j-1;a_1',a_2',\dotsc]$. 
The elements must satisfy $a_k'\ge 1$, and it 
might be that conditions such as $a_i\le \phi_i$ are not preserved under such a change.  
For example, $a_1$ increases in the terminating continued fractions
\begin{equation*}
\frac{3}{4}=[0;1,3]\quad\text{and}\quad -\frac{3}{4}=-1+\frac{1}{4}=[-1;4].
\end{equation*}
Thus it seems likely that $E$ is not symmetric (i.e., that 
$-E\ne E$).
\end{remark}

\begin{remark}
Consider $E_\textup{K}$, with the subscript chosen in honor of Khinchin, 
as the smallest possible zero-measure set of Khinchin's Theorem.
It must be symmetric. Therefore it is likely that $E_\textup{K}\ne E$. 
Note that $E_\textup{K}$ is not a unique set, but that it depends on $\set{b_n}$.
\end{remark}

%%%%%%%%%%%%%%%%%%%%%%%%%%%%%%%%%%%%%%
%%%%%%%%%%%%%%%%%%%%%%%%%%%%%%%%%%%%%%

%\bibliographystyle{amsplain}

\end{document}